\pdfoutput=1
\documentclass[a4paper]{article}
\usepackage{natbib}

\usepackage{graphicx}
\let\cite=\citep
\setcitestyle{square}
\usepackage{amsmath}
\usepackage{amsfonts,amssymb,amsthm}
\usepackage{mathrsfs}
\usepackage{a4wide}

\usepackage{ifpdf}

\ifpdf
\usepackage[bookmarks=false,%
pdfstartview=FitH,linkbordercolor={0.5 1 1},%
citebordercolor={0.5 1 0.5},unicode,%
hyperindex]{hyperref}%
\else
\usepackage{breakurl}                          
\fi

\DeclareMathOperator{\mes}{mes}

\theoremstyle{plain}
\newtheorem{theorem}{Theorem}
\newtheorem{definition}{Definition}

\newtheorem{lemma}{Lemma}

\newtheoremstyle{remark}
{3pt}
{3pt}
{\rmfamily}
{\parindent}
{\bfseries}
{.}
{.5em}
{}

\theoremstyle{remark}

\begin{document}

\date{}

\author{E. Yarovaya}

\title{Operator Equations of Branching Random Walks\thanks{This work was
supported by the Russian Foundarion for Science (project no.~14-21-00162).}}

\markboth{E. Yarovaya}{Operator Equations of Branching Random Walks}

\maketitle

\begin{abstract}
Consideration is given to the continuous-time supercritical branching
random walk over a multidimensional lattice with a finite number of
particle generation sources of the same intensity both with and without
constraint on the variance of jumps of random walk underlying the process.
Asymptotic behavior of the Green function and eigenvalue of the evolution
operator of the mean number of particles under source intensity close to
the critical one was established.\medskip

\noindent\textbf{Keywords:}  branching random walks, Green function,
convolution-type operator, multipoint perturbations, positive
eigenvalues.\medskip

\noindent\textbf{AMS Subject Classification:} 60J80, 60J35, 62G32
\end{abstract}

\section{Introduction}

The continuous-time branching random walk (BRW) with finite number of the
branching sources situated at the points of the multidimensional lattice
$\mathbb{Z}^{d}$, $d\ge1$, is considered.  BRW relies on the symmetrical
space-homogeneous irreducible random walk on $\mathbb{Z}^{d}$, $d\ge1$.  In
the recent decades, a sufficiently great number of publications was devoted
to the continuous-time BRW on $\mathbb{Z}^{d}$ (see, for example,
\cite{ABY98-1:e, AB00:e,VTY03:e,VT04:e,YarBRW:e,YTV10:e}).  It seems that
\cite{Y15-DAN:e} was the first publication to establish that the number of
positive eigenvalues in the discrete spectrum of the operator describing the
evolution of the mean numbers of particles and their multiplicity depend not
only on the source intensity, but also on their spatial configuration.  These
results were stated in detail in \cite{Y16-MCAP:e}. A definition of a weakly
supercritical BRWs whose discrete spectrum has a unique positive eigenvalue
was introduced in \cite{Y15-DAN:e}.  This case is of importance because the
fact of uniqueness of the positive eigenvalue facilitates significantly
investigation of the particle fronts \cite{MolYar12-2:e}.  A condition under
which the supercritical BRW is weakly supercritical was given in
\cite{Y15-DAN:e}.

We outline the paper content.  Section~\ref{S:statement} gives the basic
definitions.  It is assumed that the intensities of the branching sources are
identical and equal to $\beta\in\mathbb{R}$.  The mean numbers of particles
at an arbitrary point of the lattice obey the differential equations
treatable as the operator equations in the Banach spaces (see, for example,
\cite{Y13-PS:e}).  Here $\beta_{c}$ denotes the minimal value of the source
intensity such that for $\beta>\beta_{c}$ the spectrum of the evolutionary
operator of mean numbers of particles $\mathscr{H}_{\beta}$ contains positive
eigenvalues implying the exponential growth in the numbers of particles both
in an arbitrary bounded domain on $\mathbb{Z}^{d}$ and over the entire
lattice $\mathbb{Z}^{d}$ (see, for example, \cite{YTV10:e,Y11-SBRW:e}).
Section~\ref{S:WSBRW} is devoted to analysis of behavior of the Green
function $G_{\lambda}$ of the random walk transition probabilities
(Theorems~\ref{As-G} and~\ref{As-G-Ht}), as well as to determination of the
asymptotic estimates of the leading eigenvalue of the operator
$\mathscr{H}_{\beta}$ under $\beta\to\beta_{c}$ in the case of an arbitrary
number of sources $N\ge1$ both under finite (Theorem~\ref{T-EV}) and infinite
variance of jumps of random walk (Theorem~\ref{As-L-N}). Notice that in the
cases (i) and (ii)  of Theorem~\ref{As-L-N} corresponding to the recurrent
random walk under infinite variance of jumps one succeeds to determine an
explicit dependence on the number of sources~$N$.  The proofs of
Theorems~\ref{As-G-Ht} and~\ref{As-L-N} are given in Sections~\ref{S:As-G-Ht}
and~\ref{S:As-L-N}, respectively.

\section{Description of the Model}\label{S:statement}

\sloppy Let $A=(a(x,y))_{x,y\in\mathbb{Z}^d}$ be the matrix of transition
intensities of random walk, ${a(x,y)\ge0}$ for $x\ne y$, $a(x,x)<0$, where
$a(x,y)=a(y,x)=a(0,y-x)=a(y-x)$ and $ \sum_{z} a(z)=0$.  Let also $A$ be an
irreducible matrix, that is, for each $z\in\mathbb{Z}^{d}$ there exists a set
of vectors $z_{1},z_{2},\dots,z_{k}\in\mathbb{Z}^{d}$ such that
$z=\sum_{i=1}^{k}z_{i}$ and $a(z_{i})\neq0$ for $i=1,2,\dots,k$.  The
branching mechanism in the sources is independent of the walk and defined by
an infinitesimal generating function $f(u):=\sum_{n=0}^\infty b_n u^n$ where
$b_n\ge0$ for $n\ne 1$, \,$b_1<0$ and $\sum_{n} b_n=0$.  It is assumed that
each of the particles evolves independently of the rest of particles.  We
assume also that there exists the first derivative of the generating function
$\beta_1:=f'(1)<\infty$, that is, the first moment of the direct particle
offsprings is finite, and denote for brevity $\beta:=\beta_1$. For the
further exposition, it suffices only to assume that there exists the first
moment $\beta$.  However, we note that the condition for finiteness of all
moments, that is, $\beta_r:=f^{(r)}(1)<\infty$ for all $r$, is used in the
method-based proofs of the limit theorems on behavior of the numbers of
particles in BRW (see, for example, \cite{YarBRW:e}).

\fussy In the BRW models \cite{Y12_MZ:e}, multipoint perturbations of the
generator of symmetrical random walk $\mathscr{A}$ arise which in the case of
identical intensity of the sources are given by
\begin{equation}\label{E:osnH} \mathscr{H}_{\beta}=\mathscr{A}+
\beta\sum_{i=1}^{N}\varDelta_{x_{i}},
\end{equation}
where $x_{i}\in \mathbb{Z}^{d}$, $\mathscr{A}:l^{p}(\mathbb{Z}^{d})\to
l^{p}(\mathbb{Z}^{d})$, $p\in[1,\infty]$, is a symmetrical operator generated
by the matrix $A$ and obeying the formula
\[
(\mathscr{A}
u)(z):=\sum_{z'\in\mathbb{Z}^{d}}a(z-z')u(z'),
\]
$\varDelta_{x}=\delta_{x}\delta_{x}^{T}$, and $\delta_{x}=\delta_{x}(\cdot)$
denotes the column vector on the lattice assuming unit value at the point $x$
and zero value at the rest of points. The perturbation $\beta
\sum_{i}\varDelta_{x_{i}}$ of the linear operator $\mathscr{A}$ may give rise
to occurrence in the spectrum of the operator $\mathscr{H}_{\beta}$ of
positive eigenvalues, the multiplicity of each such eigenvalue not exceeding
the number of the summands $N$ in the last sum \cite{Y12_MZ:e,Y16-MCAP:e}.

The multipoint perturbations of the generator of symmetrical random walk
$\mathscr{A}$ like \eqref{E:osnH} occur in the operator equations for the
moments of particle numbers.  For example, let $\mu_{t}(y)$ be the number of
particles at the time instant $t$ at the point $y$.  Then, the condition that
at the initial time instant $t=0$ the system consists of a single particle
situated at the point $x$ is equivalent to the equality
$\mu_{0}(y)=\delta(x-y)$.  At that, the total number of particles on the
lattice obeys the equality $\mu_{t}=\sum_{y\in \mathbb{Z}^{d}}\mu_{t}(y)$.
Denote by $m_{1}(t,x,y)=\mathsf{E}_{x}(\mu_{t}(y))$ the expectation of the
number of particles at the time instant $t$ at the point $y$, provided that
$\mu_{0}(y)\equiv \delta(x-y)$, that is, at the initial time instant the
system had one particle at the point $x$.  As was shown in
\cite{Y12_MZ:e,Y13-PS:e}, the evolution of $m_{1}(t,x,y)$ obeys the operator
equation in the space $l^{2}(\mathbb{Z}^{d})$:
\[
\frac{d\,m_{1}(t,x,y)}{d\,t} =(\mathscr{H}_{\beta}m_{1}(t,\cdot,y))(x),
\qquad m_{1}(0,x,y)=\delta(x-y).
\]
Evolution of the mean number of particles
$m_{1}(t,x)=\mathsf{E}_{x}(\mu_{t}(y))$ (total size of the population) over
the entire lattice (see, for example, \cite{YarBRW:e}) satisfies the operator
equation in the corresponding space $l^{\infty}(\mathbb{Z}^{d})$:
\[
\frac{d\,m_{1}(t,x)}{d\,t} =(\mathscr{H}_{\beta}m_{1}(t,\cdot))(x),
\qquad m_{1}(0,x)=1.
\]
Now we notice that the issue of the rate of growth or decrease of the mean
number of particles $m_{1}(t,x,y)$, is tightly bound to the spectral
properties of the operator $\mathscr{H}_{\beta}$.  For example, if the
operator $\mathscr{H}_{\beta}$ has the maximal eigenvalue \mbox{$\lambda>0$,}
then $m_{1}(t,x,y)$ grows at infinity as $e^{\lambda t}$. Denote now by
$p(t,x,y)$ the transition probability of the random walk.  Clearly, the
function $p(t,x,y)$ is determined by the transition intensities $a(x,y)$
(see, for example, \cite{GS:e,YarBRW:e}).  Then the Green function of the
operator $\mathscr{A}$ is representable as the Laplace transform of the
transition probability $p(t,x,y)$:
\begin{equation}\label{E:Gosn}
G_\lambda(x,y):=\int_0^\infty e^{-\lambda t}p(t,x,y)\,dt,\qquad\lambda\geq
0.
\end{equation}
In what follows, the alternative representation of the function
$G_\lambda(x,y)$ will be useful (see, for example, \cite{YarBRW:e}):
\begin{equation}\label{Grin_lambda}
G_\lambda(x,y)=\frac{1}{(2\pi)^d} \int_{ [-\pi,\pi ]^{d}}
\frac{e^{i(\theta, y-x)}} {\lambda-\phi(\theta)}\,d\theta,\qquad
x,y\in\mathbb{Z}^{d},~\lambda \ge 0,
\end{equation}
or, equivalently,
\begin{equation}\label{Grin_lambda_cos}
G_\lambda(x,y)=\frac{1}{(2\pi)^d} \int_{ [-\pi,\pi ]^{d}}
\frac{\cos{(\theta, y-x)}} {\lambda-\phi(\theta)}\,d\theta,\qquad
x,y\in\mathbb{Z}^{d},~\lambda \ge 0,
\end{equation}
For the random walk, the function $G_{0}(x,y)$ has a simple probabilistic
sense of the mean number of hits of the particle at the point $y$ in time
$t\to\infty$ for the process starting from the point $x$.  The asymptotics of
behavior of the mean number of particles $m_{1}(t,x,y)$ for $t\to\infty$ is
also can be expressed in terms of the function $G_\lambda(x,y)$ (see, for
example, \cite{YarBRW:e}).  The same reference shows that analysis of BRW
depends essentially on whether $G_{0}:=G_{0}(0,0)$ is finite or infinite.  If
the condition for finiteness of the variance of jumps
\begin{equation}\label{E:findisp} \sum_{z} |z|^2
a(z)<\infty
\end{equation}
is satisfied (here and below $|z|$ denotes the Euclidean norm of the
vector~$z$), then $G_{0}=\infty$ for $d=1$ or $d=2$ and $G_{0}<\infty$ for
$d\ge 3$ (see, for example, \cite{ABY00:e,YarBRW:e}).  If the asymptotic
equality
\begin{equation}\label{E:infindisp}
a(z)\sim\frac{H\left(\frac{z}{|z|}\right)}{|z|^{d+\alpha}},\qquad\alpha\in(0,2),
\end{equation}
is satisfied for all sufficiently large in norm $z\in\mathbb{Z}^{d}$, where
$H(\cdot)$ is continuous, positive, and symmetrical function on the sphere
$\mathbb{S}^{d-1}:=\{z\in\mathbb{R}^{d} :  |z|=1\}$, then $G_{0}=\infty$ in
the case of $d=1$ and $\alpha \in [1,2)$, and of $d=1$, and $G_{0}$ is finite
in the cases $\alpha \in (0,1)$ or $d\geq 2$ and $\alpha \in (0,2)$
\cite{Y13-CommStat:e}.  In distinction to \eqref{E:findisp}, condition
\eqref{E:infindisp} leads to divergence of the series $\sum_{z} |z|^2 a(z)$
and thereby to infinity of variance of jumps.

We present the necessary information about the discrete spectrum of the
evolutionary operator $\mathscr{H}_{\beta}$.  Denote by
$\beta_{c}:=\beta_{c}(N)$, where $N\geq 1$, the minimal value of source
intensity such that for $\beta>\beta_{c}$ the spectrum of the operator
$\mathscr{H}_{\beta}$ contains the positive eigenvalues.

As was shown in \cite{Y16-MCAP:e}, for the BRW based on symmetrical,
space-homo\-geneous, and irreducible random walk, either of
conditions~\eqref{E:findisp} or~\eqref{E:infindisp} is satisfied.  If
$G_{0}=\infty$, then $\beta_{c}(N)=0$ for $N\geq 1$.  If $G_{0}< \infty$,
then $\beta_{c}(1)=G_{0}^{-1}$ for $N=1$ and $0<\beta_{c}(N)<G_{0}^{-1}$ for
$N\geq 2$.

In the case of $G_{0}< \infty$ for $N=2$, the quantity $\beta_{c}$ was
calculated in \cite{Y12_MZ:e} under condition \eqref{E:findisp}
\begin{equation}\label{E:bcN2}
\beta_{c}=(G_{0}+\widetilde{G}_{0})^{-1},
\end{equation}
where $\widetilde{G}_{0}=G_{0}(x_{1},x_{2})$.  However, we notice that, as
was shown in \cite{Y16-MCAP:e}, condition \eqref{E:findisp} in
\cite{Y12_MZ:e} is nonessential and the equality \eqref{E:bcN2} remains valid
even under condition \eqref{E:infindisp}.

Additional information on the structure of the discrete spectrum of the
operator $\mathscr{H}_{\beta}$ can be extracted from the statement in
\cite{Y15-DAN:e} that for $N\geq 2$ and $\beta>\beta_{c}$ the operator
$\mathscr{H}_{\beta}$ can have at most $N$ positive eigenvalues counting
their multiplicity
\begin{equation}\label{E:eigen}
\lambda_{0}(\beta)>\lambda_{1}(\beta)\geq\cdots\geq\lambda_{N-1}(\beta)>0,
\end{equation}
the leading eigenvalue $\lambda_{0}(\beta)$ being simple, that is, having the
unit multiplicity.  Moreover, there exists a value of $\beta_{c_{1}}>
\beta_{c}$ such that for $\beta\in(\beta_{c},\beta_{c_{1}})$ the operator
$\mathscr{H}_{\beta}$ has a unique eigenvalue $\lambda_{0}(\beta)$.

\section{Weakly Supercritical BRW}\label{S:WSBRW}

The concept of weakly supercritical BRW was introduced in \cite{Y15-DAN:e}.

\begin{definition}\label{D-new-1}
If there exists $\varepsilon_{0}>0$ such that for $\beta\in
(\beta_{c},\beta_{c}+\varepsilon_{0})$ the operator $\mathscr{H}_{\beta}$ has
one (counting multiplicity) positive eigenvalue $\lambda_{0}(\beta)$
satisfying the condition $\lambda_{0}(\beta)\to 0$ for $\beta \downarrow
\beta_{c}$, then the supercritical BRW is called weakly supercritical for
$\beta$ close to $\beta_{c}$.
\end{definition}

As was established in \cite{Y15-DAN:e,Y16-MCAP:e}, for $\beta \downarrow
\beta_{c}$ each supercritical BRW is weakly supercritical.  Of special
interest for the weakly supercritical branching random walks are the
asymptotics of the Green function \eqref{E:Gosn} and the eigenvalue
$\lambda_{0}(\beta)$ for the evolutionary operator \eqref{E:osnH} for
$\beta \downarrow \beta_{c}$, that is, for $\beta\to\beta_{c}$,
$\beta>\beta_{c}$.

As was shown in \cite{YarBRW:e,MolYar12:e}, the following two assertions are
valid under the condition \eqref{E:findisp}.

\begin{theorem}\label{As-G}
If $\lambda\downarrow 0$, then the following asymptotic equalities take
place:
\begin{itemize}
\item[\rm(i)] $G_{\lambda}\sim\gamma_{1}\sqrt{\pi}(\sqrt{\lambda})^{-1}$
    for $d=1$,

\item[\rm(ii)] $G_{\lambda}\sim -\gamma_{2}\ln \lambda$ for $d=2$,

\item[\rm(iii)] $G_{\lambda}- G_{0}\sim
    -2\sqrt{\pi}\gamma_{3}\sqrt{\lambda}$ for $d=3$,

\item[\rm(iv)] $G_{\lambda}-G_{0}\sim\gamma_{4}\lambda\ln {\lambda}$ for
    $d=4$,

\item[\rm(v)] $G_{\lambda}- G_{0}\sim-\gamma_{d}\lambda$ for $d\ge 5$,
\end{itemize}
where $\gamma_{d}$ is some positive constant.
\end{theorem}

\begin{theorem}\label{T-EV}
The eigenvalue $\lambda_{0}(\beta)$ of the operator $\mathscr{H}_{\beta}$
for $\beta \downarrow \beta_{c}$ has the following asymptotic behavior:
\begin{itemize}
\item[\rm(i)] $\lambda_{0}(\beta)\sim c_{1}(N\beta)^{2}$ for $d=1$,

\item[\rm(ii)] $\lambda_{0}(\beta)\sim e^{-c_{2}/(N\beta)}$ for $d=2$,

\item[\rm(iii)] $\lambda_{0}(\beta)\sim c_{3}(\beta-\beta_{c})^{2}$ for
    $d=3$,

\item[\rm(iv)] $\lambda_{0}(\beta)\sim
    c_{4}(\beta-\beta_{c})\ln^{-1}((\beta - \beta_{c}) ^{-1})$ for $d=4$,

\item[\rm(v)] $\lambda_{0}(\beta)\sim c_{d}(\beta-\beta_{c})$ for $d\ge 5$,
\end{itemize}
where $c_{i}$, $i\in \mathbb{N}$, are some positive constants.
\end{theorem}

The case of infinite variance of the random walk jumps, to which condition
\eqref{E:infindisp} gives rise, is less studied, and the following
Theorems~\ref{As-G-Ht} and~\ref{As-L-N} are devoted to its analysis.  We
notice that Theorem~\ref{T-EV} for $\mathbb{Z}^{d}$ is proved along the same
lines as Theorem~\ref{As-L-N}.  An analog of Theorem~\ref{T-EV} for
$\mathbb{R}^{d}$ can be found in \cite{CKMV09:e}.

\begin{theorem}\label{As-G-Ht}
Let $\alpha\in (0,2)$.  If $\lambda\downarrow 0$, then there are the
following asymptotic equalities:
\begin{itemize}
\item[\rm(i)]
    $G_{\lambda}\sim\gamma_{d,\alpha}\lambda^{\frac{1-\alpha}{\alpha}}$ for
    $d=1$, $\alpha\in(1,2)$,

\item[\rm(ii)] $G_{\lambda}\sim -\gamma_{d,\alpha}\ln \lambda$ for $d=1$,
    $\alpha=1$,

\item[\rm(iii)] $G_{\lambda}-
    G_{0}\sim-\gamma_{d,\alpha}\lambda^{\frac{d-\alpha}{\alpha}}$ for
    $d=1$, $\alpha\in(\frac{1}{2},1)$ or $d=2$, $\alpha\in(1,2)$ or $d=3$,
    $\alpha\in(\frac{3}{2},2)$,

\item[\rm(iv)] $G_{\lambda}-G_{0}\sim\gamma_{d,\alpha}\lambda\ln\lambda$
    for $d=1$, $\alpha=\frac{1}{2}$ or $d=2$, $\alpha=1$ or $d=3$,
    $\alpha=\frac{3}{2}$,

\item[\rm(v)] $G_{\lambda}- G_{0}\sim-\gamma_{d,\alpha}\lambda$ for $d=1$,
    $\alpha\in(0,\frac{1}{2})$ or $d=2$, $\alpha\in(0,1)$ or $d=3$,
    $\alpha\in(0,\frac{3}{2})$ or $d\ge4$, $\alpha\in(0,2)$,
\end{itemize}
where $\gamma_{d,\alpha}$ is some positive constant for each dimension $d$ of
the lattice $\mathbb{Z}^d$.
\end{theorem}

\begin{theorem}\label{As-L-N}
Let $\alpha\in (0,2)$ and $N\geq 1$.  Then, the following asymptotic
representations are valid for the function $\lambda_{0}(\beta)$ under
$\beta \downarrow \beta_{c}$.
\begin{itemize}
\item[\rm(i)] $\lambda_{0} (\beta)\sim c_{d,
    \alpha}(N\beta)^{\frac{\alpha}{\alpha-1}}$ for $d=1$, $\alpha\in(1,2)$,

\item[\rm(ii)] $\lambda_{0}(\beta)\sim e^{-c_{d,\alpha}/(N\beta)}$ for
    $d=1$, $\alpha=1$,

\item[\rm(iii)] $\lambda_{0}(\beta)\sim
    c_{d,\alpha}(\beta-\beta_{c})^{\frac{\alpha}{d-\alpha}}$ for $d=1$,
    $\alpha\in(\frac{1}{2},1)$ or $d=2$, $\alpha\in(1,2)$ or $d=3$,
    $\alpha\in(\frac{3}{2},2)$,

\item[\rm(iv)] $\lambda_{0}(\beta)\sim e^{W(-c_{d,
    \alpha}(\beta-\beta_{c}))}$ for $d=1$, $\alpha=\frac{1}{2}$ or $d=2$,
    $\alpha=1$ or $d=3$, $\alpha=\frac{3}{2}$,

\item[\rm(v)] $\lambda_{0}(\beta)\sim c_{d, \alpha}(\beta-\beta_{c})$ for
    $d=1$, $\alpha\in(0,\frac{1}{2})$ or  $d=2$, $\alpha\in(0,1)$ or $d=3$,
    $\alpha\in(0,\frac{3}{2})$ or  $d\ge4$, $\alpha\in(0,2)$,
\end{itemize}
where $c_{d, \alpha}$ is some positive constant (for each fixed values of the
parameter $\alpha$ and dimension $d$ of the lattice $\mathbb{Z}^d$), and
$W(x)$ is the lower branch of the Lambert \mbox{$W$-function}\footnote{The
Lambert function $W(z)$ of the complex variable $z$ is determined as solution
of the equation $W e^{W}=z$ (see, for example, \cite{CGHJK:96}).}, see
Fig.~\ref{F-Lambert},
satisfying the condition $W(x)\to -\infty$ for $x\uparrow 0$.
\end{theorem}

\begin{figure}[!htbp]
\begin{center}
\includegraphics[width=0.4\textwidth]{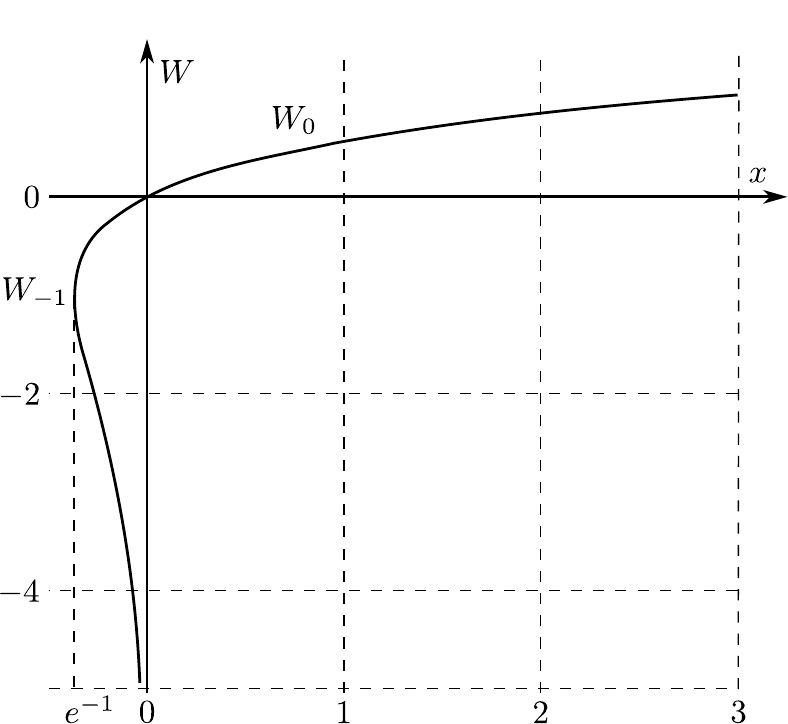}
\caption{Plot of the Lambert $W$-function}\label{F-Lambert}
\end{center}
\end{figure}

\section{Proof of Theorem~\ref{As-G-Ht}}\label{S:As-G-Ht}

Before proving Theorem~\ref{As-G-Ht}, we formulate an auxiliary statement
which is a nonsopisticated consequence of the Tauberian Theorems~2 and~4 from
\cite[vol.~2, Ch.~XIII, \S~5]{Fe:e} and Problem~16 from \cite[vol.~2,
Ch.~XIII, \S 11]{Fe:e}.  We recall that the function $L(t)>0$ determined
under sufficiently large values of $t$ is called the \emph{slowly varying at
infinity} if $\frac {L(tx)}{L(t)}\to1$ for $t\to\infty$ and each fixed $x>0$.

\begin{lemma}\label{L:Taub}
Let $u(t)$ be a positive function summable on $[0,\infty)$, and
\[
\omega(\lambda)=\int_{0}^{\infty}e^{-\lambda t}u(t)\,dt,\qquad \lambda >0
\]
be its Laplace transform.  Let also $L(t)$ be a function slowly varying at
infinity.  Then, the following statements are valid:

\begin{itemize}

\item[\rm(i)] If $0 \le \rho <\infty$, then
\begin{equation}\label{E:Taub}
\omega(\lambda)\sim\frac{1}{\lambda^{\rho}}L(\frac{1}{\lambda})\quad\text{for}\quad \lambda\to 0
\end{equation}
if and only if
\begin{equation}\label{E:Taub1}
\int_{0}^{t}u(s)\,ds\sim\frac{1}{\Gamma(\rho+1)}t^{\rho}L(t)\quad
\text{for}\quad t\to\infty.
\end{equation}

\item[\rm(ii)] If $0<\rho <\infty$ and for the function $u(t)$ there exists
    a function $v(t)$ monotone over some interval $[t_{0},\infty)$ such
    that $u(t)\sim v(t)$ for $t\to\infty$, that is, $u(t)/v(t)\to 1$ for
    $t\to\infty$, then equality \eqref{E:Taub} takes place if and only if
\begin{equation}\label{E:int u}
u(t)\sim\frac{1}{\Gamma(\rho)}t^{\rho-1}L(t)\quad\text{for}\quad
t\to\infty.
\end{equation}
\end{itemize}
\end{lemma}

We proceed now to proving Theorem~\ref{As-G-Ht}.  It was established in
\cite{AMV:ArXiv14} that for $\alpha\in(0,2)$ and $d\geq1$ and for each
$x,y\in\mathbb{Z}^d$ the asymptotic equality holds
\begin{equation}\label{p_T}
p(t, x, y) \sim \frac{h_{\alpha, d}}{t^{\frac{d}{\alpha}}}\quad\text{for}\quad t\to\infty,
\end{equation}
where $h_{\alpha, d}>0$ is a constant.  From this relation and representation
\eqref{E:Gosn} follows the fact that $G_{0}=\infty$ in the case of $d=1$ and
$\alpha \in [1,2)$, and the variable $G_{0}$ is finite in the cases of $d=1$
and $\alpha \in (0,1)$ or $d\geq 2$ and $\alpha \in (0,2)$.  These relations
for $G_{0}$ may be established without using equality~\eqref{p_T}
(see~\cite{Y13-CommStat:e}).  We begin the proof from the cases (i) and (ii)
where $G_{0}=\infty$.  Denote $p(t):=p(t,0,0)$ and assume that $u(t)\equiv
p(t)$, $v(t)=\frac{h_{\alpha, d}}{t^{\frac{d}{\alpha}}}$.  Since the function
$v(t)$ is monotone, in virtue of \eqref{p_T} the function $u(t)$ is
asymptotically monotone.

\emph{Case \rm{(i)}.} Assume that $\rho=1-\frac{1}{\alpha}$ and
$L(t)=h_{\alpha, 1}$, where by the condition $\alpha\in (1,2)$.  By virtue of
statement~(ii) of Lemma~\ref{L:Taub}, for the function
$G_{\lambda}\equiv\omega (\lambda)$ in this case there exists the asymptotic
equality $G_{\lambda} \sim\gamma_{1,\alpha}\lambda^{\frac{1-\alpha}{\alpha}}$
for $\lambda\to 0$, where $\gamma_{1,\alpha}=h_{\alpha,
1}\Gamma(1-\frac{1}{\alpha})$, which proves the theorem for the case (i).

\emph{Case \rm{(ii)}.} Here $\alpha=1$ and, therefore,
$\rho=1-\frac{1}{\alpha}=0$. Now we make use of statement~(i) of
Lemma~\ref{L:Taub}.  For that we first need to determine the asymptotics of
the integral in the left side of \eqref{E:Taub1}.  By virtue of \eqref{p_T},
the equality $v(t)=h_{1,1}t^{-1}$ is valid and, therefore,
\[
\int_{0}^{t}u(s)\,ds\sim h_{1,1}\ln t \quad\text{for}\quad t\to\infty.
\]
Now by using statement~(i) of Lemma~\ref{L:Taub} we establish that
$G_{\lambda}\sim\gamma_{1,\alpha}\ln \lambda$ for $\lambda\to 0$, where
$\Gamma(1)=1$ and $\gamma_{1,1}=-h_{1, 1}$, which proves the theorem for the
case (ii).

To prove the theorem for the cases (iii)--(v) where $G_{0}<\infty$, we need
an auxiliary relation.  By representing $G_{\lambda}$ as
\begin{multline*}
G_{\lambda}=\int_0^\infty e^{-\lambda t}p(t)\,dt=\lambda
\int_0^\infty e^{-\lambda t}\left(\int_0^t
p(s)\,ds\right)\,dt\\=\lambda \int_0^\infty e^{-\lambda
t}\left(G_{0}-\int_t^\infty
p(s)\,ds\right)\,dt=G_{0}-\lambda \int_0^\infty e^{-\lambda
t}\left(\int_t^\infty p(s)\,ds\right)\,dt,
\end{multline*}
we obtain
\begin{equation}\label{Gl-taub}
G_{0}-G_{\lambda}=\lambda \int_0^\infty e^{-\lambda
t}\left(\int_t^\infty p(s)\,ds\right)\,dt,
\end{equation}
where in virtue of \eqref{p_T} the asymptotic equality
\begin{equation}\label{Gl-taub-2} f(t):=\int_t^\infty p(s)\,ds \sim
\frac{h_{\alpha, d}}{\frac{d}{\alpha}-1}t^{1-\frac{d}{\alpha}}\quad\text{for}\quad
t\to\infty
\end{equation}
takes place.  For the cases (iii) and (iv), we assume that $u(t)\equiv f(t)$
and $v(t)= h_{\alpha,
d}\left(\frac{d}{\alpha}-1\right)^{-1}t^{1-\frac{d}{\alpha}}$.

\emph{Case \rm{(iii)}.} In this case, the estimate $1-\frac{d}{\alpha}>-1$ is
valid under all relevant combinations of the parameters $d$ and $\alpha$, and
in~\eqref{Gl-taub} the outer integral diverges for $\lambda=0$. Additionally,
the inequality $\rho>0$ is satisfied in this case for $\rho$ defined by
$\rho-1=1-\frac{d}{\alpha}$.  Also, $L(t)\equiv \Gamma(\rho) h_{\alpha, d}$.
Consequently, the statement~(ii) of Lemma~\ref{L:Taub} is applicable for
estimation of the asymptotics of the integral in the right side of
\eqref{Gl-taub}.  It follows from this statement that
\[
G_{0}-G_{\lambda}=\lambda \int_0^\infty e^{-\lambda
t}u(t)\,dt \sim
\gamma_{d,\alpha}\lambda^{\frac{d-\alpha}{\alpha}}\quad\text{for}\quad \lambda\to 0,
\]
where $\gamma_{d,\alpha}=\Gamma(\rho) h_{\alpha, d}$, which proves the
theorem for the case (iii).

\emph{Case \rm{(iv)}.} In this case, the equality
$\rho-1=1-\frac{d}{\alpha}=-1$ takes place for all relevant combinations of
the parameters $d$ and $\alpha$, and, correspondingly, $\rho=0$.  Therefore,
we use statement~(i) of Lemma~\ref{L:Taub}.  For that we need to know the
asymptotics of integral \eqref{E:Taub1}.  In virtue of \eqref{Gl-taub-2}, the
function $v(t)$ under all considered conditions is given by $v(t)=h_{\alpha,
d}t^{-1}$, and then
\[
\int_{0}^{t}u(s)\,ds\sim h_{\alpha, d}\ln t\quad\text{for}\quad t\to\infty.
\]
By using the statement~(i) of Lemma~\ref{L:Taub} we obtain that
$G_{0}-G_{\lambda}\sim\gamma_{d,\alpha}\ln \lambda$ for $\lambda\to 0$, where
$\gamma_{d,\alpha}=-\Gamma(1)h_{d, \alpha}$, which proves the theorem for the
case (iv).

\emph{Case \rm{(v)}.} For the outer integral in \eqref{Gl-taub} to be finite
and uniformly bounded under all sufficiently small $\lambda>0$, it suffices
in virtue of the inequality
\[
\int_0^\infty e^{-\lambda
t}\left(\int_t^\infty p(s)\,ds\right)\,dt \leq \int_0^\infty \left(\int_t^\infty p(s)\,ds\right)\,dt,
\]
that the function $f(t)$ as defined in \eqref{Gl-taub-2} be integrable over
$[0,\infty)$.  In its turm, by virtue of \eqref{Gl-taub-2} this takes place
for $1-\frac{d}{\alpha}<-1$ or, which is the same, $d>2\alpha$. Consequently,
under condition (v) the integral in \eqref{Gl-taub} is finite, and for
$\lambda \to 0$ we get
\begin{equation}\label{E:gamma0}
G_{0}-G_{\lambda}=\gamma_{d,\alpha}\lambda, \quad \text{where} \quad
\gamma_{d,\alpha}=\int_0^\infty \left(\int_t^\infty p(s)\,ds\right)\,dt
\end{equation}
which completes the proof of theorem for the case (v).

By using Theorem~\ref{As-G-Ht}, one can obtain asymptotic representations for
the function $\lambda_{0}(\beta)$ for $\beta \downarrow \beta_{c}$ in the
case of an arbitrary fixed number of sources~$N$.

\section{Proof of Theorem~\ref{As-L-N}}\label{S:As-L-N}

Prior to proving the theorem, we prove an auxiliary lemma.

\begin{lemma}\label{L:Ginfty}
Let $\phi(\theta)$, where $\theta \in [-\pi,\pi]^d$ is a continuous
function such that $\phi(\theta)<0$ for $\theta\neq 0$ and $\phi(0)=0$, and
let
\begin{equation*}
\Theta(\lambda)= \int_{[-\pi,\pi]^d} \frac{1}{\lambda
-\phi(\theta)}d\theta \to \infty\quad\text{for}\quad\lambda \downarrow 0.
\end{equation*}
Let additionally $f(\theta)$ and $h(\theta)$ be functions continuous on
$[-\pi,\pi]^d$ and satisfying the conditions $f(0)=0$ and $h(0)=1$.  Then,
for $\lambda \downarrow 0$
\begin{equation}\label{E:2}
\frac{F(\lambda)}{\Theta(\lambda)}\to 0,
\end{equation}
\begin{equation}\label{E:3}
\frac{H(\lambda)}{\Theta(\lambda)}\to 1,
\end{equation}
where
\[
F(\lambda)=\int_{[-\pi,\pi]^d} \frac{f(\theta)}{\lambda - \phi(\theta)}d\theta, \quad
H(\lambda)=\int_{[-\pi,\pi]^d} \frac{h(\theta)}{\lambda - \phi(\theta)}d\theta.
\]
\end{lemma}

\begin{proof}
Relation~\eqref{E:3} follows from \eqref{E:2} if under the given function
$h(\theta)$ one takes $f(\theta)=h(\theta)-1$ and notes that in this case
$H(\lambda)=F(\lambda)+\Theta(\lambda)$.

We prove relation~\eqref{E:2}.  Take an arbitrary $\varepsilon>0$ and use it
to define $\rho=\rho(\varepsilon)>0$ such that $|f(\theta)|\leq \varepsilon$
for $\theta \in \mathbb{S}_{\rho}:=\{ x:  \|x\| \leq \rho\}$. For the already
obtained $\rho$, denote by $\kappa_{\rho}$ the variable
\[
\kappa_{\rho}:=\mes\left\{ [-\pi,\pi]^d \setminus \mathbb{S}_{\rho}\right\}\cdot\max_{\theta\in [-\pi,\pi]^d \setminus
 \mathbb{S}_{\rho}}\left|\frac{f(\theta)}{\phi(\theta)}\right|.
\]
We note that in virtue of continuity of the functions $\phi(\theta)$ and
$f(\theta)$, as well as of the fact that $\phi(\theta)\neq 0$ for $\theta\neq
0$, the quantity $\kappa_{\rho}$ is finite and, moreover, for $\lambda>0$ the
following inequalities hold:
\begin{equation}\label{E:UNQ}
\left|\int_{[-\pi,\pi]^d \setminus \mathbb{S}_{\rho}} \frac{f(\theta)}{\lambda -
\phi(\theta)}\,d\theta \right|\leq \int_{[-\pi,\pi]^d \setminus \mathbb{S}_{\rho}}
\left|\frac{f(\theta)}{\phi(\theta)}\right|\,d\theta\leq \kappa_{\rho}.
\end{equation}
Estimate now the quantity $\left|\frac{F(\lambda)}{\Theta(\lambda)}\right|$:
\[
\left|\frac{F(\lambda)}{\Theta(\lambda)}\right|=\left|\frac{\int_{[-\pi,
\pi]^d} \frac{f(\theta)}{\lambda -
\phi(\theta)}d\theta}{\Theta(\lambda)}\right|\leq
\frac{\left|\int_{\mathbb{S}_{\rho}} \frac{f(\theta)}{\lambda -
\phi(\theta)}d\theta\right|}{\Theta(\lambda)}+\frac{\left|\int_{[-\pi,\pi]^
d\setminus \mathbb{S}_{\rho}} \frac{f(\theta)}{\lambda -
\phi(\theta)}d\theta\right|}{\Theta(\lambda)}.
\]
For the first summand, we get
\[
\left|\int_{\mathbb{S}_{\rho}}
\frac{f(\theta)}{\lambda -
\phi(\theta)}d\theta\right|\leq \int_{\mathbb{S}_{\rho}}
\frac{|f(\theta)|}{\lambda -
\phi(\theta)}d\theta\leq \varepsilon \int_{\mathbb{S}_{\rho}}
\frac{1}{\lambda -
\phi(\theta)}d\theta \leq \varepsilon\int_{[-\pi,\pi]^d}
\frac{1}{\lambda -
\phi(\theta)}d\theta=\varepsilon \Theta(\lambda)
\]
from which and \eqref{E:UNQ} we conclude that
\[
\left|\frac{F(\lambda)}{\Theta(\lambda)}\right|\leq \varepsilon+\frac{\kappa_{\rho}}{\Theta(\lambda)},
\]
and therefore,
\[
\limsup_{\lambda\downarrow 0}\left|\frac{F(\lambda)}{\Theta(\lambda)}\right|\leq \varepsilon.
\]
Then \eqref{E:2} takes place in virtue of arbitrariness of $\varepsilon>0$.
\end{proof}

Proceed now to proving Theorem~\ref{As-L-N}.  Assume that
\begin{equation}\label{E:G}
\Gamma(\lambda)=\left[\begin{array}{ccc}
G_{\lambda}(x_{1},x_{1})   &\cdots&G_{\lambda}(x_{1},x_{N})   \\
G_{\lambda}(x_{2}, x_{1})  &\cdots&G_{\lambda}(x_{2},x_{N})  \\
\dots &\dots&\dots\\
G_{\lambda}(x_{N},x_{1})  &\cdots  &G_{\lambda}(x_{N},x_{N})\\
\end{array}\right].
\end{equation}
Denote by $\gamma_{0}(\lambda)\ge\gamma_{1}(\lambda)\geq\dots\geq\gamma_{N-
1}(\lambda)$ the eigenvalues of the matrix $\Gamma(\lambda)$ arranged in the
descending order which are determined from the equation
\begin{equation}\label{E:detGamma}
\det\left(\Gamma(\lambda)-\gamma I\right)=0.
\end{equation}
As was shown in \cite[Lemma~5]{Y16-MCAP:e}, between the two first
eigenvalues there indeed exists the strict inequality
$\gamma_{0}(\lambda)>\gamma_{1}(\lambda)$, and therefore,
\[
\gamma_{0}(\lambda)>\gamma_{1}(\lambda)\geq\dots\geq\gamma_{N-1}(\lambda).
\]
As was shown in \cite{Y16-MCAP:e}, in this case for each $\beta$ the leading
eigenvalue $\lambda_{0}(\beta)$ of the operator $\mathscr{H}_{\beta}$ is
found from the equation
\begin{equation}\label{E:gamma}
\gamma_{0}(\lambda)=\frac{1}{\beta},
\end{equation}
which is significant in what follows.

By the definition~\eqref{E:Gosn}, the function $G_{\lambda}$ is monotone in
$\lambda$, and therefore, two cases, $G_{\lambda}\to \infty$ and
$G_{\lambda}\leq C<\infty$ for $\lambda\to 0$, are possible, where $C$ is
some positive constant.

Consider first the case of $G_{\lambda}\to \infty$ for $\lambda\to 0$.  As in
Theorem~\ref{As-G-Ht}, this situation is possible only under conditions (i)
and (ii).  In this case, under fixed $x_{i}$ and $x_{j}$ in virtue of
\eqref{Grin_lambda_cos}
\[
G_\lambda(x_{i},x_{j})=\frac{1}{(2\pi)^d} \int_{ [-\pi,\pi ]^{d}}
\frac{h_{ij}(\theta)} {\lambda-\phi(\theta)}\,d\theta,\qquad
x,y\in\mathbb{Z}^{d},~\lambda \ge 0,
\]
where $h_{ij}(\theta)=\cos{(\theta, x_{j}-x_{i})}$, and since
$h_{ij}(0)=1$, the relation
\begin{equation}\label{E:GxG0}
\frac{G_{\lambda}(x_{i},x_{j})}{G_{\lambda}}\to 1\quad\text{for}\quad \lambda\to 0
\end{equation}
is valid in virtue of Lemma~\ref{L:Ginfty}.

Represent the equality \eqref{E:detGamma} as
\[
\det\left(\Gamma(\lambda)-\gamma I\right)=G^{N}_{\lambda}
\det\left(\frac{1}{G_{\lambda}}\Gamma(\lambda)-\frac{\gamma }{G_{\lambda}}I\right)=0.
\]
Denote by $\tilde{\gamma}_{i}(\lambda)$ the eigenvalues of the matrix
$\frac{1}{G_{\lambda}}\Gamma(\lambda)$.  Then,
\begin{equation}\label{E:eigentg}
\tilde{\gamma}_{i}(\lambda)=\frac{\gamma_{i}(\lambda)}{G_{\lambda}},
\end{equation}
where the quantities $\tilde{\gamma}_{i}(\lambda)$ are determined from the
equation
\[
\det\left(\frac{1}{G_{\lambda}}\Gamma(\lambda)-\tilde{\gamma}_{i}(\lambda)I
\right)=0
\]
In virtue of \eqref{E:G} and~\eqref{E:GxG0},
\begin{equation}\label{E:MU}
\frac{1}{G_{\lambda}}\Gamma(\lambda)\to \tilde{\Gamma}=\left[\begin{array}{ccc}
1  &\cdots&1   \\
1  &\cdots&1  \\
\dots &\dots&\dots\\
1 &\cdots  &1\\
\end{array}\right]\quad\text{for}\quad \lambda\to 0,
\end{equation}
As a simple calculation demonstrates, the matrix $\tilde{\Gamma}$ has one
simple eigenvalue (i.e. of unit multiplicity) equal to $N$ and $N-1$
coinciding zero eigenvalues.  In virtue of the representation \eqref{E:MU},
by the theorem on continuous dependence of the eigenvalues on a
matrix~\cite{Gant:e} we then obtain that
\begin{equation}\label{E:eigenGamma}
\tilde{\gamma}_{0}(\lambda)\to N.
\end{equation}
It follows in this case from~\eqref{E:eigentg} and~\eqref{E:eigenGamma}
that
\[ \gamma_{0}(\lambda)\sim NG_{\lambda}\quad\text{for}\quad \lambda\to 0,
\]
that is, by virtue of \eqref{E:gamma}\, $\lambda=\lambda_{0}(\beta)$ obeys
the equation
\begin{equation}\label{E-Glb}
NG_{\lambda}(1+\varphi_{0}(\lambda))=\frac{1}{\beta},
\end{equation}
where $\varphi_{0}(\lambda)$ is a function satisfying
$\varphi_{0}(\lambda)\to 0$ for $\lambda\to 0$.

\emph{Case \rm{(i)}.} In virtue of assertion (i) of Theorem~\ref{As-G-Ht} the
function $G_{\lambda}$ is representable as
\begin{equation}\label{E:la}
G_{\lambda}=\gamma_{1,\alpha}\lambda^{\frac{1-\alpha}{\alpha}}(1+\varphi
(\lambda)),
\end{equation}
where $\varphi(\lambda)$ is some function for which $\varphi(\lambda)\to 0$
under $\lambda\to 0$.  It follows from that and \eqref{E-Glb} that in this
case the quantity $\lambda=\lambda_{0}(\beta)$ is given by
\[
\gamma_{1,\alpha}\lambda^{\frac{1-\alpha}{\alpha}}\cdot
(1+\varphi(\lambda))(1+\varphi_{0}(\lambda))=\frac{1}{N\beta},
\]
whence it follows that
\[
\lambda_{0}(\beta)=\left(\frac{1}{\gamma_{1,\alpha} N\beta}
\left[\frac{1}{(1+\varphi(\lambda_{0}(\beta)))(1+\varphi_{0}(\lambda_{0}(\beta)))}\right]
\right)^{\frac{\alpha}{1-\alpha}}.
\]
Since $\lambda_{0}(\beta)\to 0$ for $\beta\to 0$, the bracketed expression
tends to 1 under $\beta\to 0$.  By assuming that $c_{1,\alpha}=(\gamma_{1,
\alpha})^{-\frac{\alpha}{1-\alpha}}$, we obtain
\[
\lambda_{0}(\beta)\sim c_{1, \alpha}(N\beta)^{\frac{\alpha}{\alpha-1}}
\]
for $\beta\to 0$, which proves the theorem for the case (i).

\emph{Case \rm{(ii)}.} In virtue of assertion (ii) of Theorem~\ref{As-G-Ht},
the representation
\begin{equation}\label{E:lb}
G_{\lambda}=  -\gamma_{1,1}(\ln \lambda)(1+\varphi(\lambda))
\end{equation}
is valid, where again $\varphi(\lambda)$ is some function satisfying
$\varphi(\lambda)\to 0$ for \mbox{$\lambda\to 0$}.  Then, as in the preceding
case we establish in virtue of \eqref{E-Glb} that the quantity
$\lambda=\lambda_{0}(\beta)$ is given by
\[
-\gamma_{1,1}(\ln \lambda)(1+\varphi(\lambda))(1+\varphi_{0}(\lambda))=
\frac{1}{N\beta},
\]
whence it follows that
\[
\ln\lambda_{0}(\beta)=-\frac{1}{\gamma_{1,1} N\beta}
\left[\frac{1}{(1+\varphi(\lambda_{0}(\beta)))(1+\varphi_{0}(\lambda_{0}(\beta)))}\right].
\]
Again, since $\lambda_{0}(\beta)\to 0$ for $\beta\to 0$, the bracketed
expression tends to $1$ for $\beta\to 0$.  By assuming that
$c_{1,1}=(\gamma_{1,1})^{-1}$, we establish
\[
\lambda_{0}(\beta)\sim
e^{-\frac{c_{1,1}}{N\beta}}
\]
for $\beta\to 0$, which proves the theorem for the case (ii).

Now we go to the proof in the case of
\[
G_{\lambda}\leq C<\infty\quad\text{for}\quad\lambda\to 0.
\]
We notice that this inequality can take place only if conditions (iii)--(v)
of Theorem~\ref{As-G-Ht} are satisfied.  In this case, in virtue of
\eqref{Gl-taub} the following equalities are valid for each pair of
subscripts $i$ and $j$:
\begin{equation}\label{Gl-taubx}
G_{0}(x_{i},x_{j})-G_{\lambda}(x_{i},x_{j})=\lambda \int_0^\infty e^{-\lambda
t}\left(\int_t^\infty p(s,x_{i},x_{j})\,ds\right)\,dt,
\end{equation}
whence, as in the proofs of statements (iii)--(v) of Theorem~\ref{As-G-Ht},
it follows for $\lambda\to 0$ that
\[
G_{0}(x_{i},x_{j})-G_{\lambda}(x_{i},x_{j})\sim \psi_{d,\alpha}(\lambda),
\]
where $\psi_{d,\alpha}(\lambda)$ is one of the functions
$\gamma_{d,\alpha}\lambda^{\frac{d-\alpha}{\alpha}}$ or
$-\gamma_{d,\alpha}\ln \lambda$ or $\gamma_{d,\alpha}\lambda$.  We underline
that here the right side of the asymptotic equality, the function
$\psi_{d,\alpha}(\lambda)$, is independent of $x_{i}$ and $x_{j}$.

In this case, the matrix $\Gamma(\lambda)$ is representable as
\[
\Gamma(\lambda)=\Gamma(0)+\psi_{d,\alpha}(\lambda)Q(\lambda),
\]
where the particular form of the function $\psi_{d,\alpha}(\lambda)$ is
determined by assertions (iii)--(v) of Theorem~\ref{As-G-Ht}, and
$Q(\lambda)$ in the limit is again the matrix consisting only of units.
Then, the equation \eqref{E:detGamma} for $\gamma_{0}(\lambda)$ can be
rearranged in
\[
\det\left(\Gamma(0)+\psi_{d,\alpha}(\lambda)Q(\lambda)-\gamma_{0}(\lambda)I\right)=0.
\]
As can be seen from this equation, for $\lambda\to 0$ not only the relation
$\gamma_{0}(\lambda)\to \gamma_{0}(0)$ takes place, but also, since the
leading eigenvalue of the matrix $\Gamma(0)$ is simple, by the theorem of
smooth dependence of the simple eigenvalues under smooth perturbations of a
matrix~\cite{Kato:e} there exists a number $\kappa>0$ such that
\[
\gamma_{0}(\lambda)=\gamma_{0}(0)+\kappa\psi_{d,\alpha}(\lambda)(1+\tilde{\varphi}(\lambda)),
\]
where $\tilde{\varphi}(\lambda)\to 0$ for $\lambda\to 0$.\footnote{We notice
that, according to the theory of perturbations of linear operators
\cite{Kato:e} the number $\kappa$ is defined by the structure of the matrices
$\Gamma(0)$ and $\Gamma'_{\lambda}(0)$, and therefore, depends in particular
on $d$, $\alpha$ and $N$.}
To determine $\lambda_{0}(\beta)$, with regard for
\[
\gamma_{0}(0)=\frac{1}{\beta_{c}},
\]
we get from \eqref{E:gamma}
\[
\frac{1}{\beta_{c}}+\kappa\psi_{d,\alpha}(\lambda)(1+\tilde{\varphi}(\lambda))=\frac{1}{\beta}
\]
or
\begin{equation}\label{E:kappa}
\kappa\psi_{d,\alpha}(\lambda)(1+\tilde{\varphi}(\lambda))=\frac{1}{\beta}-\frac{1}{\beta_{c}}.
\end{equation}

\emph{Case \rm{(iii)}.} In this case, according to statement (iii) of
Theorem~\ref{As-G-Ht} the function $\psi_{d,\alpha}(\lambda)$ for
$\lambda\to\lambda_{0}$ is given by
\[
\psi_{d,\alpha}(\lambda)= -\gamma_{d,\alpha}\lambda^{\frac{d-\alpha}{\alpha}}(1+\varphi(\lambda)),
\]
where $\varphi(\lambda)$ is some function satisfying $\varphi(\lambda)\to
0$ for $\lambda\to 0$.  Then for $\lambda=\lambda_{0}(\beta)$ one can put
down the equation
\[
-\kappa\gamma_{d,\alpha}\lambda^{\frac{d-\alpha}{\alpha}}(1+\varphi(\lambda))(1+\tilde{\varphi}(\lambda))=\frac{1}{\beta}-\frac{1}{\beta_{c}}.
\] Whence it follows that
\[
(\lambda_{0}(\beta))^{\frac{d-\alpha}{\alpha}}=
\frac{\beta-\beta_{c}}{\kappa\gamma_{d,\alpha}\beta^{2}_{c}}
\left[\frac{\beta_{c}}{\beta}\cdot\frac{1}{\left(1+\varphi(\lambda_{0}(\beta))\right)\cdot \left(1+\tilde{\varphi}(\lambda_{0}(\beta))\right)}\right].
\]
As above, the bracketed expression tends to 1 for $\beta\to\beta_{c}$.
Therefore, by assuming that $c_{d,\alpha}=(\kappa\gamma_{d,\alpha}\beta^{2
}_{c})^{-\frac{\alpha} {d-\alpha}}$ we get
\[
\lambda_{0}(\beta)\sim c_{d,
\alpha}(\beta-\beta_{c})^{\frac{\alpha}{d-\alpha}}\quad\text{for}\quad\beta\to\beta_{c},
\]
which completes the proof of theorem for the case (iii).

\emph{Case \rm{(iv)}.} In virtue of statement (iv) of Theorem~\ref{As-G-Ht},
in this case $\psi_{d,\alpha}(\lambda)$ for $\lambda\to\lambda_{0}$ is given
by
\begin{equation}\label{E:ld}
\psi_{d,\alpha}(\lambda)=\gamma_{d,\alpha}\lambda\ln\lambda(1+\varphi(\lambda)),
\end{equation}
where $\varphi(\lambda)\to 0$ for $\lambda\to 0$.  In virtue of
\eqref{E:kappa} we then get the equality
\[
\kappa\gamma_{d,\alpha}\lambda\ln\lambda(1+\varphi(\lambda))(1+\tilde{\varphi}(\lambda))=\frac{1}{\beta}-\frac{1}{\beta_{c}},
\]
whence it follows that
\[
\lambda_{0}(\beta)\ln\lambda_{0}(\beta)
=
\frac{\beta-\beta_{c}}{-\kappa\gamma_{d,\alpha}\beta^{2}_{c}}
\left[\frac{\beta_{c}}{\beta}\cdot\frac{1}{\left(1+\varphi(\lambda_{0}(\beta))\right)\cdot \left(1+\tilde{\varphi}(\lambda_{0}(\beta))\right)}\right].
\]
Since in this case $\lambda_{0}(\beta)\to 0$ for $\beta\to\beta_{c}$, the
bracketed expression tends to 1 for $\beta\to\beta_{c}$.  Therefore, we get
$\lambda_{0}(\beta)\ln \lambda_{0}(\beta)\sim -c_{d,
\alpha}(\beta-\beta_{c})$, where
$c_{d,\alpha}=(\kappa\gamma_{d,\alpha}\beta^{2}_{c})^{-1}$ or, which is the
same,
\[
\lambda_{0}(\beta)\ln
\lambda_{0}(\beta)= -c_{d,
\alpha}(\beta-\beta_{c})+(\beta-\beta_{c})\tilde{\varphi}(\beta),
\]
where $\bar{\varphi}(\beta)$ is some function such that
$\bar{\varphi}(\beta)\to 0$ for $\beta\to\beta_{c}$.  In this case,
\[
\lambda_{0}(\beta)=e^{W(-c_{d, \alpha}(\beta-\beta_{c})+(\beta-\beta_{c})
\bar{\varphi}(\beta))},
\]
where $W(x)$ is the lower branch of the Lambert $W$-function \cite{CGHJK:96}.
Extract the ``main part'' in the right side of the obtained equality, for
which purpose represent this equality as
\begin{equation}\label{E:eW}
\lambda_{0}(\beta)=e^{W(-c_{d, \alpha}(\beta-\beta_{c}))}e^{\bar{W}(\beta)},
\end{equation}
where
\[
\bar{W}(\beta)=W(-c_{d, \alpha}(\beta-\beta_{c})+(\beta-\beta_{c})
\bar{\varphi}(\beta))-W(-c_{d, \alpha}(\beta-\beta_{c})).
\]
By making use of the fact that the derivative of the Lambert $W$-function
is given by
\[
W'(x)=\frac{W(x)}{x(1+W(x))}, \qquad x\neq 0,
\]
(see~\cite{CGHJK:96}), by the mean-value theorem we get the expression
\[
\bar{W}(\beta)=W'(\theta)\bigl(-c_{d, \alpha}(\beta-\beta_{c})+(\beta-\beta_{c})
\bar{\varphi}(\beta)-(-c_{d, \alpha}(\beta-\beta_{c}))\bigr),
\]
where $\theta$ is some number satisfying the inequality
\[
-c_{d, \alpha}(\beta-\beta_{c})\leq \theta \leq -c_{d, \alpha}(\beta-\beta_{c})+(\beta-\beta_{c})
\bar{\varphi}(\beta).
\]
Then,
\[
\bar{W}(\beta)=\frac{W(\theta)}{(1+W(\theta))}\cdot \frac{(\beta-\beta_{c})
\bar{\varphi}(\beta)}{\theta},
\]
where $\theta \to 0$ for $\beta\to\beta_{c}$.  Therefore, also
$\frac{W(\theta)}{1+W(\theta)}\to 1$ for $\beta\to\beta_{c}$ because
$W(\theta)\to -\infty$ for $\theta \to 0$.  Consequently,
$\tilde{W}(\beta)\to 0$ for $\beta\to \beta_{c}$.

Therefore, in equality \eqref{E:eW} we have $e^{\tilde{W}(\beta)}\to 1$ for
$\beta\to \beta_{c}$ and then, $\lambda_{0}(\beta) \sim e^{W(-c_{d,
\alpha}(\beta-\beta_{c}))}$, which completes the proof of the theorem in the
case~(iv).

\emph{Case \rm{(v)}.} As follows from statement (v) of Theorem~\ref{As-G-Ht},
in this case the function $\psi_{d,\alpha}(\lambda)$ for
$\lambda\to\lambda_{0}$ is given by
\[
\psi_{d,\alpha}(\lambda)= -\gamma_{d,\alpha}\lambda(1+\varphi(\lambda)),
\]
where $\varphi(\lambda)$ is some function satisfying $\varphi(\lambda)\to
0$ for $\lambda\to 0$.  Then, by virtue of \eqref{E:kappa} we get the
equality
\[
-\kappa\gamma_{d,\alpha}\lambda(1+\varphi(\lambda))(1+\tilde{\varphi}(\lambda))=\frac{1}{\beta}-\frac{1}{\beta_{c}},
\]
whence it follows that
\[
\lambda_{0}(\beta)
=
\frac{\beta-\beta_{c}}{\kappa\gamma_{d,\alpha}\beta^{2}_{c}}
\left[\frac{\beta_{c}}{\beta}\cdot\frac{1}{\left(1+\varphi(\lambda_{0}(\beta))\right)\cdot \left(1+\tilde{\varphi}(\lambda_{0}(\beta))\right)}\right].
\]
The bracketed expression here, as above, tends to $1$ for $\beta\to\beta
_{c}$.  Therefore, by assuming that $c_{d,\alpha}=(\kappa\gamma_{d,\alpha}
\beta^{2}_{c})^{-1}$ we obtain that $\lambda_{0}(\beta)\sim \mbox{$c_{d,
\alpha}(\beta-\beta_{c})$}$ for $\beta\to\beta_{c}$, which completes the
proof of case~(v) and also the entire theorem.

\section*{Acknowledgments}

This study has been carried out at Steklov Mathematical Institute of Russian
Academy of Sciences, and was supported by the Russian Science Foundation,
project no.~14-21-00162.

\fussy


\end{document}